\documentclass[12pt]{article}
\title{On totally $k$-closed nilpotent groups}
\usepackage{authblk}
\author{Dmitry Churikov}
\affil{Novosibirsk State Technical University, Novosibirsk, Russia}

\usepackage{amsmath,amssymb,latexsym,amsthm,enumerate,color}

\newtheorem{theorem}{Theorem}[section]
\newtheorem*{theoremA}{Theorem A}
\newtheorem*{theoremB}{Theorem B}
\newtheorem{lemma}[theorem]{Lemma}

\newtheorem*{problem}{Problem}

\theoremstyle{definition}

\def\Om{\Omega}

\def\mZ{{\mathbb Z}}

\def\Sym{{\rm Sym}}
\def\Orb{{\rm Orb}}

\def\Syl{{\rm Syl}}

\begin{document}
\maketitle

\begin{abstract}
A group $G$ is said to be totally $k$-closed for a positive integer $k$ if, in each of its faithful permutation representations on a set $\Om^k$, $G$ is the largest subgroup of the symmetric group $\Sym(\Om)$ that preserves every $k$-orbit in the induced action on the set $\Om\times\dots\times \Om=\Om^k$. We prove that for $k\geq1$, every finite nilpotent group with Sylow subgroups of orders at most $p^k$ for all primes $p$ dividing $|G|$ is totally $k$-closed if and only if it does not contain an elementary abelian subgroup $\mZ_p^k$ for every prime~$p$.
\end{abstract}

\section{Introduction}\label{intro}

In 1969, Wielandt~\cite[Definition 5.3]{Wielandt1969} introduced the concept of the $k$-closure of a permutation group $G\leq\Sym(\Omega)$ for each positive integer $k$. The \emph{$k$-closure} $G^{(k)}$ of $G$ is the group of all permutations $g\in\Sym(\Om)$ that preserve each $G$-orbit in the induced $G$-action on ordered $k$-tuples from $\Om$. A permutation group $G$ is called \emph{k-closed} if $G^{(k)}=G$. Also Wielandt~\cite[Theorem 5.8]{Wielandt1969} proved that for $k\geq 2$
\begin{equation}
	G\leq G^{(k)}\leq G^{(k-1)}\leq\Sym(\Om).
\end{equation}
Thus if $G$ is $(k-1)$-closed, then it is also $k$-closed.

In 2016, D.~F.~Holt\footnote{\texttt{mathoverflow.net/questions/235114/2-closure-of-a-permutation-group}} proposed a concept of a totally closed (abstract) group that does not rely on a particular permutation representation. For a positive integer $k$, a group $G$ is said to be \emph{totally $k$-closed} if  $G^{(k)}=G$ whenever $G$ is faithfully represented as a permutation group on $\Om$. The only totally $1$-closed group is the trivial group, while Abdollahi and Arezoomand \cite[Theorem 2]{AA} showed that a finite nilpotent group is totally $2$-closed if and only if it is cyclic, or it is a direct product of a generalized quaternion group and a cyclic group of odd order. Since then totally $2$-closed finite groups were well described, see~\cite{AAT, AIPT}.

C.~E.~Praeger and the author studied totally $k$-closed abelian groups. By the fundamental theorem for finite abelian groups, each nontrivial finite abelian group $G$ can be decomposed into the direct product of certain cyclic groups, which are called the \emph{invariant factors} of $G$. It was discovered that the number of invariant factors $n(G)$ of an abelian group $G$ determines whether or not $G$ is totally $k$-closed.

\begin{theorem}\label{CPr}{\rm \cite[Theorem 1.2]{CPr}}
Let $G$ be a finite abelian group with $|G|>1$. Then $G$ is totally $(n(G)+1)$-closed, but is not totally $n(G)$-closed.
\end{theorem}

In this short note we study the following

\begin{problem}{\rm \cite[Problem 1]{CPr}}
For $k>2$ determine all finite nilpotent groups $G$ that are totally $k$-closed.
\end{problem}

The main results are as follows.

\begin{theoremA}
For $k\geq2$, a finite nilpotent group is totally $k$-closed if and only if all its Sylow subgroups are totally $k$-closed.
\end{theoremA}

\begin{theoremB}
Suppose that $G$ is a finite nilpotent group with Sylow $p$-subgroups of orders at most $p^k$ for all primes $p$ dividing $|G|$. Then $G$ is totally $k$-closed if and only if it does not contain an elementary abelian subgroup $\mZ_p^k$ for every prime~$p$.
\end{theoremB}

\section{Proofs}

The sufficiency in Theorem A easily follows from the main result of~\cite{ChNl}.

\begin{theorem}\label{ChNl}{\rm \cite[Theorem 1.1]{ChNl}}
	If $G$ is a finite nilpotent permutation group, and $k\geq 2$, then
	$$
	G^{(k)}=\prod_{P\in\operatorname{Syl}(G)} P^{(k)}.
	$$
\end{theorem}
Now suppose that $G$ is a totally $k$-closed finite nilpotent group, $G_p\in\Syl(G)$ for each $p$, and $G_q$ is not totally $k$-closed for some $q$ and $k\geq 2$. Then there exists a faithful permutation representation $H_q$ of $G_q$ on a set $\Omega_q$ such that $H_q^{(k)} > H_q$. For each other prime divisor $p$ of $|G|$ put $\Omega_p = G_p$, and consider $G_p$ acting regularly on $\Omega_p$ by right multiplication resulting permutation representation $H_q$. Thus $G$ acts faithfully on $\bigcup_{p\in\pi(G)} \Omega_p$, and  Theorem~\ref{ChNl} implies that
$$
G^{(k)}=\prod_{p\in\pi(G)} G_p^{(k)}=G_q^{(k)}\times\prod_{\substack{p\in\pi(G) \\ p\neq q}} G_p^{(k)},
$$
which is not equal to $G$, because $G_q^{(k)}>G_q$ and for every $p\in\pi(G), p\neq q$ the group $G_p$ is $k$-closed as a regular group. Thus, $G$ is not totally $k$-closed, and the proof of Theorem~A is complete.\medskip

In order to prove Theorem~B we need some observations on bases of $p$-groups. Let $G$ be a permutation group on a set $\Omega$. A subset of $\Omega$ is said to be a \emph{base} for $G$ if its pointwise stabilizer in $G$ is trivial. The minimal size of a base for $G$ is called the \emph{base number} of $G$ and denoted by~$b(G)$.

\begin{lemma}\label{Base}
	If $G$ is a permutation group of order $p^k$, then $b(G)\leq k$. Moreover, if $G$ is nonabelian, then $b(G)\leq k-1$. The bounds are strict.
\end{lemma}

\begin{proof}
	Construct a decreasing chain of stabilizers for $G$:
	$$G> G_{\alpha_1}> G_{\alpha_1\alpha_2}> \ldots> G_{\alpha_1\dots\alpha_{d}}=1.$$
	Since $|G|=p^k$, the chain must end within a maximum of $k$ steps, so $b(G)\leq d\leq k$.

	Now let $G$ be nonabelian. Every permutation group is a subgroup of the direct product of its transitive constituents. Since $G$ is nonabelian, there exists a nonabelian transitive constituent $G^\Delta$, with $\Delta\in\Orb(G)$. If $\alpha\in\Delta$, then $G_\alpha$ is a $p$-group too, so $|G_\alpha|=p^i$, with $0\leq i\leq k-1$.

	If $i \leq k-2$, then previous reasonings imply that $b(G_\alpha) \leq k-2$, and consequently, $b(G) \leq k-1$.
	
	If $i = k-1$, then $|G:G_\alpha|=p$, so $|\Delta|=p$ and $G^\Delta$ is abelian, which is not the case.
	
	For $k\geq 1$, the groups $\mZ_2^k=\langle (1~2), (3~4),\ldots,(2k-1~2k) \rangle$ and $D_8\times \mZ_2^k=\langle (1~2~3~4), (1~3), (5~6), (7~8),\ldots,(2k+3~2k+4) \rangle$ have bases $\{2,4,6,\ldots,2k\}$ and $\{1,2,6,8,\ldots,2k+4\}$ respectively, so the bounds are strict.
\end{proof}

The classical result of Wielandt establishes a connection between the base number of a group and its upper bound to be $k$-closed.

\begin{theorem}\label{W}{\rm\cite[Theorem 5.12]{Wielandt1969}}\quad  Let $G\leq \Sym(\Omega)$ and $k\geq1$, and suppose that $\alpha_1,\ldots,\alpha_{k}\in\Omega$ such that $G_{\alpha_1\ldots \alpha_{k}}=1$. Then $G^{(k+1)}=G$. In other words, if $b(G)\leq k$, then $G$ is $(k+1)$-closed.
\end{theorem}

An easy combination of Theorem~\ref{W} and Lemma~\ref{Base} results the following

\begin{lemma}\label{NA}
Every nonabelian group of order at most $p^k$ is totally $k$-closed.
\end{lemma}

Now we prove Theorem~B. Let $G$ be a nilpotent group with Sylow $p$-subgroups of orders at most $p^k$. Theorem~A implies that $G$ is totally $k$-closed if and only if all its Sylow subgroups are totally $k$-closed. A group of order at most $p^k$ is totally $k$-closed if and only if it is not $\mZ_p^k$, because all such nonabelian groups are $k$-closed by Lemma~\ref{NA} and all such abelian groups except $\mZ_p^k$ are totally $k$-closed by Theorem~\ref{CPr}.

\subsection*{Acknowledgements}
The author would like to thank Andrey Vasil'ev for valuable comments.


\begin{thebibliography}{99}

\bibitem{Wielandt1969} H. W. Wielandt, Permutation groups through invariant relations and invariant
functions, Lecture Notes, Ohio State University, 1969. Also published in: Wielandt, Helmut, Mathematische Werke (Mathematical works) Vol. 1. Group theory. Walter de Gruyter \& Co., Berlin, 1994, 237--296.

\bibitem{AA} A. Abdollahi and M. Arezoomand, Finite nilpotent groups that coincide with
their 2-closures in all of their faithful permutation representations, J. Algebra
Appl., 17(4) 2018, 1850065.

\bibitem{AAT} A. Abdollahi, M. Arezoomand and Gareth Tracey, On finite totally $2$-closed groups. Comptes Rendus. Mathématique, vol. 360 (2022), 1001-1008.

\bibitem{AIPT} M. Arezoomand, M. Iranmanesh, C. E. Praeger and G. Tracey, Totally 2-closed finite groups with trivial Fitting subgroup, Bulletin of Mathematical Sciences, 2023, 2350004.

\bibitem{CPr} D. Churikov and C. E. Praeger, Finite totally $k$-closed groups, Trudy Instituta Matematiki i Mekhaniki UrO RAN, vol. 27, no. 1 (2021), 240-245.

\bibitem{ChNl} D.V. Churikov Structure of $k$-Closures of Finite Nilpotent Permutation Groups. Algebra Logic 60 (2021), 154-159.
\end{thebibliography}
\end{document}